\documentclass[reqno, 11pt, twoside]{amsart}

\usepackage{amsfonts,amsthm,amsmath,amssymb}
\usepackage[margin=3.5cm]{geometry}
\usepackage{mathrsfs}
\let\mathcal\mathscr
\usepackage{enumitem}
\usepackage{hyperref}
\hypersetup{
  colorlinks=true,
  citecolor=blue,
  linkcolor=blue,
  urlcolor=blue,
  pdfstartview=FitH
}

\numberwithin{equation}{section}

\newtheorem{theorem}{Theorem}[section]
\newtheorem{lemma}[theorem]{Lemma}
\newtheorem{corollary}[theorem]{Corollary}
\newtheorem{prop}[theorem]{Proposition}
\newtheorem{conj}[theorem]{Conjecture}

\theoremstyle{definition}
\newtheorem{remark}[theorem]{Remark}

\renewcommand{\leq}{\leqslant}

\renewcommand{\geq}{\geqslant}

\newcommand{\Pic}{\mathrm{Pic}}

\newcommand{\eps}{\varepsilon}
\newcommand{\PP}{\mathbb{P}}
\newcommand{\sqf}{\mathrm{sf}}

\newcommand{\RR}{\mathbb{R}}
\newcommand{\AAA}{\mathbb{A}}

\newcommand{\QQ}{\mathbb{Q}}
\newcommand{\ZZ}{\mathbb{Z}}
\newcommand{\NN}{\mathbb{N}}

\begin{document}

\title{Class numbers and integer points on some Pellian surfaces}

\author{Yijie Diao}
\address{IST Austria\\
Am Campus 1\\
3400 Klosterneuburg\\
Austria}
\email{yijie.diao@ist.ac.at}
\urladdr{https://sites.google.com/view/yijie-diao/home}
\thanks{The author would like to thank his supervisor Tim Browning for suggesting this project and many helpful conversations and useful comments. Moreover, he is grateful to Jakob Glas, Damaris Schindler, Igor Shparlinski, Matteo Verzobio, Victor Wang, Florian Wilsch and Shuntaro Yamagishi for taking their time to answer his questions and their valuable suggestions.}

\begin{abstract}
We provide an estimate for the number of nontrivial integer points on the Pellian surface $t^2 - du^2 = 1$ in a bounded region. We give a lower bound on the size of fundamental solutions for almost all $d$ in a certain class, based on a recent conjecture of Browning and Wilsch about integer points on log K3 surfaces. We also obtain an upper bound on the average of class number in this class, assuming the same conjecture.
\end{abstract}

\maketitle

\thispagestyle{empty}
\setcounter{tocdepth}{1}
\tableofcontents

\section{Introduction}

\subsection{The Pell equation}

A \textit{Pell equation} is a Diophantine equation of the form
\begin{equation}\label{Pell}
    t^2 - du^2 = 1,
\end{equation}
where $d \in \NN$ is a positive integer that is not a square.
When $d$ is a negative integer or a perfect square, the only possible integer solutions are the trivial ones $(t, u) = (\pm 1, 0)$. 
In this article, we assume that $d \geq 2$ is not a perfect square. For given $d$, it is convenient to write a solution as $\eta_d = t + u\sqrt{d}$. Dirichlet's unit theorem tells us that the set of integer solutions of (\ref{Pell}) is isomorphic to the group of integers and there exists a unique \textit{fundamental solution} $\eps_d = t_1 + u_1 \sqrt{d}$ satisfying
$\{\eta_d: \text{ solution of (\ref{Pell})} \} = \{ \pm \eps_d^n, \ n \in \ZZ \}$ and $t_1, u_1 \geq 1$. 

Note that $t_1^2 = 1 + du_1^2 \geq 1 + d$, so we obtain the lower bound
\begin{eqnarray}\label{lower-bound}
	\eps_d \geq \sqrt{d+1} + \sqrt{d}.
\end{eqnarray}
The equality holds if and only if $d + 1$ is a perfect square.

Let $h(d)$ be the class number of quadratic forms of determinant\footnote{ We use the term \textit{determinant} $d$ instead of \textit{discriminant} $\Delta$, in order to align with Hooley \cite{Hoo} and differ from the regular definition $\Delta(ax^2 + bxy + cy^2) = b^2 - 4ac$. See also \cite[Chapter 1.3.D]{Cox} for the history and connection between these two terms.}  $d$. Specifically, $h(d)$ is the number of properly primitive classes of indefinite forms $ax^2 + 2bxy + cy^2$ with determinant $d = b^2 - ac$. This unusual definition was originally due to Gauss. The advantage is that any integer can be a determinant, instead of having to meet certain criteria modulo $4$.

In order to find an upper bound of $\eps_d$, the class number formula \cite[Formula (3)]{Hoo} tells us that
\begin{eqnarray}\label{cnf}
	h(d)\log \eps_d = \sqrt{d} L_d(1), \text{ where } L_d(1) = \sum_{\substack{m = 1 \\ m \text{ odd}}}^{\infty} \Big( \frac{d}{m} \Big) \frac{1}{m}.
\end{eqnarray}
The size of $L_d(1)$ fluctuates within relatively narrow ranges. From \cite[Equation (3.11)]{Bru}, we have $L_d(1) \ll \log d$. Therefore, on using $h(d) \geq 1$, we know that
\begin{equation}\label{fs-ub}
	\log \eps_d \ll  \sqrt{d} \log d.
\end{equation}

\subsection{Density of integer points on the Pellian surface}

In this paper we shall study the following counting function of nontrivial integer solutions to the Pell equation with bounded height:
\begin{align}\label{count-func}
 	N(B) = \#\{ (t,d,u) \in \ZZ^3: t^2 - du^2 = 1, \max (|t|, |d|, |u|) \leq B, u \neq 0 \}.
\end{align}
The growth rate of $N(B)$ is closed related to the average size of fundamental solutions.

For $\alpha > 0$ and $x \geq 2$, Hooley \cite{Hoo} introduced the counting function
\begin{equation*}
    S(x, \alpha) = \# \{ \eta_d: 2 \leq d \leq x,  \eps_d \leq \eta_d \leq d^{\frac{1}{2} + \alpha} \}.
\end{equation*}
For $0 < \alpha \leq \frac{1}{2}$, he \cite[Theorem 1]{Hoo} has proved that
\begin{equation}\label{thm-Hoo}
    S(x, \alpha) \sim \frac{4\alpha^2}{\pi^2} x^{\frac{1}{2}} \log^2 x \text{, when } x \rightarrow \infty.
\end{equation}
For $\alpha > \frac{1}{2}$, Hooley has suggested a conjecture \cite[Conjecture 1]{Hoo} for the asymptotic behavior of $S(x, \alpha)$ when $x \rightarrow \infty$. In particular, his conjecture implies that
\begin{equation}\label{HooConj}
    S(x, \alpha) \ll_\alpha x^{\frac{1}{2}} \log^2 x.
\end{equation}

We are now ready to reveal our first result.
\begin{theorem}\label{thm-cf}
Let $\epsilon > 0$.
    \begin{enumerate}[label=(\roman*)]
        \item We have $$B^\frac{1}{2}(\log B)^2 \ll N(B) \ll_\epsilon B^{\frac{7}{12} + \epsilon}.$$
        \item If Hooley's Conjecture (\ref{HooConj}) is true for $\alpha = \frac{3}{2}$, then we also have $$N(B) \ll B^\frac{1}{2} (\log B)^2.$$
    \end{enumerate}
\end{theorem}
We will combine results from Fouvry and Jouve \cite{FoJo} and Reuss \cite{Reu} for the upper bound in Part(i). The lower bound will be deduced from (\ref{thm-Hoo}).

\subsection{The log K3 surface and $\AAA^1$-curve solutions}
 A smooth cubic surface $U \subset \AAA_\QQ^3$ is said to be \textit{log K3}, if there is a morphism from a smooth, projective surface $\tilde{X}$ to the completion $X$ of $U$ in $\PP_\QQ^3$, which is an isomorphism over $U$, and such that $\tilde{D} = \tilde{X} \setminus U$ is a divisor with strict normal crossings whose class in $\mathrm{Pic}_\QQ (\tilde{X})$ is $\omega_{\tilde{X}}^\lor$. In particular, it follows from the adjunction formula that $U$ is $\log$ K3, if $X$ itself is smooth over $\QQ$.


Let $U$ be a log K3 surface. We denote
\begin{eqnarray*}
	N_U(B) = \# \{ (x_1, x_2, x_3) \in U(\ZZ): \max(|x_i|) \leq B \}.
\end{eqnarray*}
When $U(\ZZ)$ contains an $\AAA^1$-curve $\phi$ that is defined over $\ZZ$, the curve contributes $\asymp B^{\frac{1}{\deg(\phi)}}$ points to $N_U(B)$. We typically expect the contribution from $\AAA^1$-curves to dominate the counting function. Hence it is natural to study the subset $U(\ZZ)^\circ$ obtained by removing the points in $U(\ZZ)$ that lie on any $\AAA^1$-curve defined over $\ZZ$. This leads us to consider the following counting function
\begin{eqnarray*}
	N_U^\circ(B) = \# \{ (x_1, x_2, x_3) \in U(\ZZ)^\circ: \max(|x_i|) \leq B \}.
\end{eqnarray*}

In a recent paper, Browning and Wilsch \cite[Conjecture 1.1]{BroWil} have proposed the following conjecture for the specified logarithmic growth  for a certain class of log K3 surfaces.
\begin{conj}[Browning--Wilsch]\label{BW-h}
	Let $U \subset \AAA^3$ be a cubic surface that is smooth and log K3 over $\QQ$ and that is defined by a cubic polynomial $f \in \ZZ[x_1, x_2, x_3]$. Denote by $\rho_U$ the Picard number of $U$ over $\QQ$ and by $b$ the maximal number of components of $\tilde{D}(\RR)$ that share a real point. Then
	$$N_U^\circ(B) \ll_U (\log B)^{\rho_U + b},$$
	when $B \rightarrow \infty$.
\end{conj}



Let $ U_\mathrm{P} = V(t^2 - du^2 - 1) \subset \AAA_\QQ^3 $ be the surface defined by the Pell equation. Surprisingly, every integer point $(t_0,d_0,u_0)$ on $U_\mathrm{P}$ lies on an $\AAA^1$-curve defined over $\ZZ$. This can be verified by the parametrization (see Zapponi \cite[Proposition 3]{Zap})
\begin{align*}
	t(z) = & \ (t_0 + 1)u_0^4z^2 + 2(t_0 + 1)u_0^2z + t_0, \\
	d(z)  = & \ (t_0 + 1)^2u_0^2z^2 + 2(t_0 + 1)^2z + d_0, \\
	u(z)  = & \ u_0^3 z + u_0.
\end{align*}
This parametrization is consistent with the lower bound for $N(B)$ in Theorem \ref{thm-cf}.

Instead of studying the original Pellian surface, we shall investigate the integer points on the Pellian equation
\begin{eqnarray}\label{Pell-3}
	t^2 - (z^2 + k) u^2 = 1,
\end{eqnarray}
for some non-zero $k \in \ZZ$.
Changing a variable $y = t + uz$ takes us to the equation 
\begin{eqnarray}\label{logK3}
	2uyz = y^2 - ku^2 - 1.
\end{eqnarray} 
We will show in Section \ref{PicU} that this equation defines a log K3 surface when $k = 3$.

If we choose $k = 1$, then the fundamental solution of (\ref{Pell-3}) can be calculated as
$$\eps_{z^2 + 1} = 2z^2 + 1 + \sqrt{z^2 + 1} \cdot 2z.$$ Therefore, the integer solution $(y, u, z)$ to the equation (\ref{logK3}) that corresponds to the fundamental solution $t_1 + u_1 \sqrt{z^2 + 1}$ to the Pell equation (\ref{Pell-3}) lies on the $\AAA^1$-curve
$$y(z) = 4z^2 + 1, \quad u(z) = 2z,$$
defined over $\ZZ$. According to the structure of integer solutions to the Pell equation, one may prove by induction that every integer point on the surface defined by (\ref{Pell-3}) lies on an $\AAA^1$-curve defined over $\ZZ$.
Indeed, one may check that the above phenomena happens whenever $k = \pm 1, \pm 2$, or $\pm 4$.

Therefore, we shall focus on the example
\begin{eqnarray}\label{Pell-2}
	t^2 - (z^2 + 3) u^2 = 1,
\end{eqnarray}
with $k = 3$, as well as the corresponding log K3 surface defined by 
\begin{eqnarray}\label{logK3-2}
	2uyz = y^2 - 3u^2 - 1.
\end{eqnarray}
We will show in Proposition \ref{A1-sol} that for any integer solution $(t,u,z)$ to the equation (\ref{Pell-2}), if $3 \nmid z$, $z^2 + 3$ is square-free and $u \neq 0$, then it does not lie on an $\mathbb{A}^1$-curve defined over $\mathbb{Z}$.

In Theorem \ref{thm-cf}, we used Hooley's heuristics on the typical size of the fundamental solutions to assess the size of the counting function $N(B)$ in (\ref{count-func}). Our next result uses a reverse process to extract information about the typical size of the fundamental solutions, assuming Conjecture \ref{BW-h}.

\subsection{An almost all lower bound for fundamental solutions}

Hooley \cite[Conjecture 2]{Hoo} has proposed the conjecture that for almost all $d$, we have $\log \eps_d \gg d^{\frac{1}{2} - \epsilon}$. However, we are still very far away from this result. In the same paper, Hooley has firstly shown that for almost all $d$, we have $\eps_d \gg d^{\frac{3}{2} - \epsilon}$. Reuss \cite[Corollary 10]{Reu} has improved Hooley's result to $\eps_d \gg d^{3 - \epsilon}$.  
	
Let $U$ be the surface defined by the equation (\ref{logK3-2}). We will prove the following theorem.
\begin{theorem}\label{thm-lb}
Let $\epsilon > 0$.
	Assume Conjecture \ref{BW-h} is true for the surface $U$. Let $d(z) = z^2 + 3$. Then for almost all $z \in \ZZ_{> 0}$ with $3 \nmid z$ and square-free $d(z)$, we have
	\begin{eqnarray*}
		\log \eps_{d(z)} \gg d(z)^{\frac{1}{8} - \epsilon}.
	\end{eqnarray*}
\end{theorem}

\begin{remark}\label{rmk-Ricci}
The classical work of Ricci \cite{Ric} shows that there is a positive proportion of $z \in \ZZ_{> 0}$ with $3 \nmid z$ such that $d(z) = z^2 + 3$ is square-free. 
\end{remark}

Golubeva \cite{Gol1} has shown that for almost all primes $p$ and $d = 5p^2$, we have $\log \eps_{d} \gg d^{\frac{1}{4}}$. However, much less result is known when $d$ is square-free.
For any positive integer $d$, there is a unique way to decompose $d = d_1 d_2^2$, where $d_1$ is square-free. We call $\sqf(d) := d_1$ the \textit{square-free part} of $d$. 
The best previous result for the lower bound of the size of fundamental solution for an infinite set of square-free $d$ is that of Yamamoto \cite[Theorem 3.2; Example (II)]{Yam}, in which $\log \eps_{d} \gg (\log d)^3$, where $d = \sqf \big((3 \cdot 2^m + 3)^2 - 8 \big)$. By combining Theorem \ref{thm-lb} and Remark \ref{rmk-Ricci}, we obtain infinitely many square-free $d$ such that $\log \eps_d \gg d^{\frac{1}{8}-\epsilon}$, under the assumption that Conjecture $\ref{BW-h}$ is true for the surface $U$.

We recall that the fundamental unit of a real quadratic field $K$ is defined as
$$\eps_K = \frac{a+b\sqrt{\Delta_K}}{2},$$
where $\Delta_K$ is the discriminant of $K$, and $(a,b)$ is the pair of smallest positive integers satisfying
$a^2 - \Delta_K \cdot  b^2 = \pm 4$. 
We shall deduce following result.

\begin{corollary}
	Assume Conjecture \ref{BW-h} is true for the surface $U$. Then there exists infinitely many real quadratic fields $K$, such that 
 \begin{eqnarray}\label{rqf}
     \log \eps_K \gg \Delta_K^{\frac{1}{8} - \epsilon},
 \end{eqnarray}
 when $\Delta_K \rightarrow \infty$.
\end{corollary}

\subsection{The average of class numbers}

Hooley \cite[Conjecture 7]{Hoo} has conjectured that
\begin{eqnarray*}
	\sum_{d \leq Z} h(d) \sim c_1 \, Z \cdot \log^2 Z,
\end{eqnarray*}
for an explicit constant $c_1 > 0$.
Unsurprisingly, we understand this conjecture as little as Hooley's other conjectures on the average size of fundamental solutions to the Pell equation.

Sarnak \cite{Sar1} has considered a similar question, but in a slightly different setting. Let $\mathcal{D} = \{ \Delta > 0: \Delta \equiv 0, 1 ( \mod 4), d \text{ not a square} \}$ be the set of positive discriminants. Let $H(\Delta)$ denote the number of inequivalent primitive binary quadratic forms $ax^2 + bxy + cy^2$ of discriminant $\Delta = b^2-4ac$, and let $\tilde{\eps}_\Delta$ be the fundamental solution of the Pellian equation $t^2 - \Delta u^2 = 4$. Let $\mathcal{D}_x = \{\Delta \in \mathcal{D}: \tilde{\eps}_\Delta \leq x\}$. He has shown that
\begin{eqnarray*}
	\frac{1}{\# \mathcal{D}_x} \sum_{\Delta \in \mathcal{D}_x} H(\Delta) = \frac{16}{35} \frac{Li(x^2)}{x} + O(x^{\frac{2}{3} + \epsilon}).
\end{eqnarray*} 
Note that this ordering is different from the usual one.

In a subsequent paper \cite{Sar2}, Sarnak studied the behavior of $H(d)$ along a thin sequence $\Delta(z) = z^2 - 4$, where the fundamental solutions are equal to $\tilde{\eps}_{\Delta(z)} = (z + \sqrt{z^2-4})/2$. He has shown that
\begin{eqnarray*}
	\sum_{z \leq Z} H(z^2 - 4) \sim c_2 \, Z^2 \cdot (\log Z)^{-1},
\end{eqnarray*} 
where $c_2 > 0$ is an explicit constant.

Let us return to the special values of $d(z) = z^2 + 3$. The average of the class number is dominated by the contributions when $z$ is a multiple of 3, namely when the integer point $(t, z, u)$ lies on an $\AAA^1$-curve defined over $\ZZ$. Using Yamamoto's Theorem \cite[Theorem 3.1]{Yam}, which will be introduced in Section \ref{sec3}, one may deduce that
\begin{eqnarray}\label{Yam-tri}
		\sum_{\substack{z \leq Z, 3 \, \nmid \, z \\ \mu^2(z^2 + 3) = 1}} h(z^2 + 3) \ll Z^2 \cdot (\log Z)^{-2}.
\end{eqnarray}
	
Let $U$ be the surface defined by the equation (\ref{logK3-2}). Conditionally, we will prove the following improvement.
\begin{theorem}\label{thm-cn}
	Assume Conjecture \ref{BW-h} is true for the surface $U$. Then we have
	\begin{eqnarray*}
		\sum_{\substack{z \leq Z, 3 \, \nmid \, z \\ \mu^2(z^2 + 3) = 1}} h(z^2 + 3) \ll Z^{\frac{9}{5}} \cdot (\log Z)^{\frac{3}{5}}.
	\end{eqnarray*}
\end{theorem}

\section{Proof of Theorem \ref{thm-cf}}

\subsection{The initial decomposition of the Pell equation}
Fouvry and Jouve \cite{FoJo} have considered the following decomposition, which can be traced back to Legendre and Dirichlet.
The Pell equation $(\ref{Pell})$ can be rearranged as
$$(t+1)(t-1) = du^2.$$
Note that $\gcd(t-1, t+1) | 2$, so there are three possibles splittings of $d$ and $u$, as follows:
\begin{itemize}
    \item If $t$ is even then $u = u_1u_2, d = d_1d_2$ and $d_1u_1^2 - d_2u_2^2 = 2$.
    \item If $t$ is odd and $4 \nmid d$, then $u = 2u_1u_2, d = d_1d_2$ and $d_1u_1^2 - d_2u_2^2 = 1$.
    \item If $t$ is odd and $4 | d$, then $u = u_1u_2, d = 4d_1d_2$ and $d_1u_1^2 - d_2u_2^2 = 1$.
\end{itemize}
We call the above splittings as the \textit{initial decomposition of the Pell equation}.

For $\theta = \pm1, \pm 2$, we consider the following function
$$N_\theta (D_1, D_2, U_1, U_2) = \# \{ (d_1, d_2, u_1, u_2) \in \ZZ^4: d_i \sim D_i, u_i \sim U_i, d_1u_1^2 - d_2u_2^2 = \theta\},$$
where the notation $n \sim N$ means $N < n \leq 2N$.
Reuss \cite[Theorem 5]{Reu} has applied the approximate determinant method to show that
\begin{align}\label{thm-Reu}
	N_\theta (D_1, D_2, U_1, U_2) \ll (D_2 U_2)^\eps \min \big((U_1U_2M)^\frac{1}{2} + U_1 + U_2, (D_1D_2M)^\frac{1}{2} + D_1 + D_2 \big), 
\end{align}
where 
$$\log M = \frac{9}{8} \frac{\log (D_1D_2) \log(U_1U_2)}{\log(D_1U_1^2)}.$$

We also need a result by Fouvry and Jouve \cite[Lemma 8]{FoJo}, which is particularly useful when $D_1$ or $D_2$ is very small.
For $D_1 \leq D_2$, we have
\begin{eqnarray}\label{thm-FJ}
	N_\theta (D_1, D_2, U_1, U_2) \ll (D_1D_2)^\eps \big((D_1D_2)^\frac{1}{2} + D_1U_2\big).
\end{eqnarray}

By combining (\ref{thm-Reu}) and (\ref{thm-FJ}), we may prove the following lemma.
\begin{lemma}\label{Reu+FJ}
	Let $\epsilon > 0$. Assume that $D_1 \leq D_2$, $D_1 D_2 \leq B$ and $D_1U_1^2 = D_2 U_2^2 = B$, then for any $\theta = \pm 1$ or $\pm 2$, we have
	$$N_\theta(D_1, D_2, U_1, U_2) \ll B^{\frac{7}{12} + \epsilon}.$$
\end{lemma}
\begin{proof}
	When $D_1 U_2 \leq B^{\frac{7}{12}}$, then (\ref{thm-FJ}) implies 
	$$N_\theta(D_1, D_2, U_1, U_2) \ll B^\epsilon (B^{\frac{1}{2}} + B^{\frac{7}{12}}) \ll B^{\frac{7}{12} + \epsilon}.$$
	
    Next we observe that if $\log (D_1D_2) = k \log B$, then 
    \begin{eqnarray}\label{log-M}
    	m: = \frac{\log M}{\log B} = \frac{9}{16} k (2 - k).
    \end{eqnarray}
    Since $0 \leq k \leq 1$, the function $m(k)$ is an increasing function. 
	
	When $D_1 U_2 > B^{\frac{7}{12}}$, then considering $D_2U_2^2 = B$, we know 
	$$D_1 D_2^{-\frac{1}{2}} = D_1U_2B^{-\frac{1}{2}} > B^{\frac{1}{12}}.$$ 
	Since $D_1D_2 \leq B$, we obtain
	$$D_2D_1^{-1} = (D_1D_2)^{\frac{1}{3}} \cdot (D_1D_2^{-\frac{1}{2}})^{-\frac{4}{3}} < B^{\frac{1}{3} - \frac{1}{9}} =    B^{\frac{2}{9}}.$$
	On the other hand, using $D_1 \leq D_2$, we know  $$D_1D_2 = (D_1^{-1}D_2)^3 \cdot (D_1D_2^{-\frac{1}{2}})^4 \geq (D_1 D_2^{-\frac{1}{2}})^4 > B^{\frac{1}{3}}.$$ 
	In other words, we have $k > \frac{1}{3}$. Since $m(k)$ is increasing when $0 \leq k \leq 1$, we have $m(k) > m(\frac{1}{3}) = \frac{5}{16}$. Hence we have
	\begin{eqnarray}\label{estimate-M}
		M > B^{\frac{5}{16}}.
	\end{eqnarray}
	Therefore, it follows that 
	$$M > B^{\frac{2}{9}} \geq D_2D_1^{-1}, \text{ which implies } (D_1D_2M)^{\frac{1}{2}} > D_2 \geq D_1.$$
	If we further assume that $k \leq \frac{2}{3}$, then by (\ref{thm-Reu}) we have
	$$N_\theta(D_1, D_2, U_1, U_2) \ll B^\epsilon (D_1D_2M)^{\frac{1}{2}} \ll B^{\frac{1}{2}(\frac{2}{3} + m(\frac{2}{3})) + \epsilon} = B^{\frac{7}{12} + \epsilon},$$
	since $m(k) + k$ is also an increasing function. 
	
The remaining case is when $D_1 U_2 > B^{\frac{7}{12}}$ and $k > \frac{2}{3}$. Note that $$U_1U_2 = (BD_1^{-1})^{\frac{1}{2}} \cdot (BD_2^{-1})^{\frac{1}{2}} = B^{1 - \frac{k}{2}}.$$ 
By (\ref{estimate-M}), we obtain that
$$M > B^{\frac{5}{16}} > B^{\frac{1}{9}} \geq (D_2D_1^{-1})^{\frac{1}{2}} = U_1U_2^{-1}, \text{ which implies }  (U_1U_2M)^{\frac{1}{2}} > U_1 \geq U_2.$$ 
We also have
$$(U_1U_2M)^{\frac{1}{2}} = B^{\frac{1}{2}(1  - \frac{k}{2} + m(k))}.$$
Note that the function 
$- \frac{k}{2} + m(k) = \frac{1}{16}k(10-9k)$
 is decreasing when $k > \frac{5}{9}$. Therefore, by the assumption $k > \frac{2}{3}$ and (\ref{thm-Reu}) we have
$$N_\theta(D_1, D_2, U_1, U_2) \ll B^\epsilon (U_1U_2M)^{\frac{1}{2}} \ll B^{\frac{1}{2}(1 - \frac{1}{3} + m(\frac{2}{3})) + \epsilon} = B^{\frac{7}{12} + \epsilon}.$$
	
	
	
	
\end{proof}
As a remark, one may find that the extreme case appears when $k = \frac{2}{3}$ and $D_1 U_2 \geq B^{\frac{7}{12}}$. This is  when $D_1D_2 = B^{\frac{2}{3}}$ and $B^{\frac{5}{18}} \leq D_1 \leq B^{\frac{1}{3}} \leq D_2 \leq B^{\frac{7}{18}}$.

\subsection{Proof of Theorem \ref{thm-cf}}
Without loss of generality, we may assume that $B$ is an integer to ensure that $2t = (\eps_d + \eps_d^{-1}) \leq B \Leftrightarrow \eps_d  \leq B$.

For the lower bound, we use the fact that
        \begin{align*}
            N(B) & \gg \, \# \{ (t,d,u) \in \ZZ^3: t^2 - du^2 = 1, 2 \leq t \leq d \leq B \} \\
            & \geq \ \# \{ d \in \ZZ: 2 \leq d \leq B, \eps_d \leq d \} \\
            & \gg B^\frac{1}{2}(\log B)^2,
        \end{align*}
        by Hooley's Theorem (\ref{thm-Hoo}) for $\alpha = \frac{1}{2}$.
        
        For the upper bound, we consider the number of solutions in dyadic intervals:
        \begin{equation}\label{dya}
            N(2B) - N(B) = M_1(B) + M_2(B),
        \end{equation}
        where
        \begin{equation*}
            M_1(B) = \# \{ (t,d,u) \in \ZZ^3: t^2 - du^2 = 1, d \sim B, t \leq 2B \},
        \end{equation*}
        and
        \begin{equation*}
            M_2(B) = \# \{ (t,d,u) \in \ZZ^3: t^2 - du^2 = 1, d \leq B, t \sim B \}.
        \end{equation*}
        By Hooley's Theorem (\ref{thm-Hoo}) for $\alpha = \frac{1}{2}$, we have 
        \begin{align}
        	M_1(B) & \leq \, \# \{ d \in \ZZ: 2 \leq d \leq 2B, \eps_d \leq 2d \} \\
            & \ll B^{\frac{1}{2}} (\log B)^2.
        \end{align}
        Note that $2t = \eta_d + \eta_d^{-1} > \eta_d = (\eps_d)^n$. Hence by (\ref{lower-bound}) we have
        \begin{align*}
        	& \ \#\{ (t,d,u) \in \ZZ^3: t^2 - du^2 = 1, d \leq B^{\frac{1}{2}}, t \sim B\} \\
        	 \leq & \ \#\{ (t,d,u) \in \ZZ^3: t^2 - du^2 = 1, d \leq B^{\frac{1}{2}}, \eta_d \leq 4B \} \\
        	 \leq & \ \sum_{d \leq B^{\frac{1}{2}}} \left[ \frac{\log 4B}{\log \eps_d} \right] \, \ll \, \sum_{d \leq B^{\frac{1}{2}}} \left[ \frac{\log 4B}{\log d} \right] \\
        	 \ll & \ \log B \cdot \frac{B^{\frac{1}{2}}}{\log B} \, = \, B^{\frac{1}{2}}. 
        \end{align*}
        This implies that $M_2(B) = M_3(B) + O(B^{\frac{1}{2}})$, where
        \begin{eqnarray*}
        	M_3(B) =  \#\{ (t,d,u) \in \ZZ^3: t^2 - du^2 = 1, B^{\frac{1}{2}} < d \leq B, t \sim B \}.
        \end{eqnarray*}    
        If we assume Hooley's Conjecture (\ref{HooConj}) is true for $\alpha = \frac{3}{2}$, we have 
        \begin{eqnarray*}
        	M_3(B) \leq  \#\{ (t,d,u) \in \ZZ^3: t^2 - du^2 = 1, d \leq B, t \leq 2d^2 \} \ll B^{\frac{1}{2}} (\log B)^2.
        \end{eqnarray*}
        By using (\ref{dya}), this completes the proof of part (2) of Theorem \ref{thm-cf}.

        Unconditionally, we may divide $D_1$ and $D_2$ into $\ll (\log B)^2$ dyadic intervals. By using the initial decomposition of the Pell equation and Lemma \ref{Reu+FJ}, we have
        \begin{align*}
        	M_2(B) 
        	\ll \sup_{\substack{D_1D_2 \leq B, D_1 \leq D_2 \\ D_1U_1^2 = D_2U_2^2 = B, \ \theta = \pm 2}} N_\theta(D_1,D_2,U_1,U_2) \cdot (\log B)^2 \ll B^{\frac{7}{12} + \epsilon}.
        \end{align*}
        This concludes the proof of Theorem \ref{thm-cf}.

\vspace{1em}
\section{Proof of Theorem \ref{thm-lb} and Theorem \ref{thm-cn}}\label{sec3}



We recall the equation (\ref{Pell-2}) from the introduction. This takes the shape
\begin{eqnarray*}
	t^2 - d(z)u^2 = 1, \text{ where } d(z) = z^2 + 3.
\end{eqnarray*}
For all $z = 3k > 0$, we have
$$\eps_{d(z)} = (6k^2 + 1) + \sqrt{d(z)} \cdot 2k.$$
Hence for each $z$ which is a multiple of $3$, the integer point $(t,z,u)$ lies on an $\AAA^1$-curve defined over $\ZZ$. 

\subsection{The $\AAA^1$-curve solutions}

The main result of this section is the following proposition.
\begin{prop}\label{A1-sol}
	Let $(t,u,z_0)$ be an integer solution of (\ref{Pell-2}) such that $3 \nmid z_0$, $d(z_0)$ square-free and $u \neq 0$. Then it does not lie on an $\AAA^1$-curve defined over $\ZZ$.
\end{prop} 

Yamamoto \cite[Theorem 3.1; Example 4 (I)]{Yam} has proved the following lower bound for the fundamental unit of a special class of real quadratic fields.
\begin{theorem}[Yamamoto]\label{thm-Yam}
   Let $K_\alpha = \QQ \big(\sqrt{\alpha^2 \pm 4p} \big)$ for a given prime $p$. Assume that $p$ splits in $K_\alpha$. Then we have 
$$\log \eps_{K_\alpha} \gg (\log \Delta_{K_\alpha})^2.$$
\end{theorem}

Let $K$ be a real quadratic field. Let $t_1 + u_1 \sqrt{d}$ be the fundamental solution of the Pell equation for $d = c^2 \Delta_K, c \in \ZZ_{>0}$. We have $a^2 - \Delta_K \cdot b^2 = 4$ for $a = 2t_1$ and $b = 2cu_1$. Therefore, for any $c \in \ZZ_{>0}$ we have
\begin{equation}\label{tool-1}
	\eps_K \leq \eps_{c^2 \Delta_K}.
\end{equation}
The following corollary gives a lower bound for the fundamental solution of (\ref{Pell-2}).
\begin{corollary}\label{cor-yam}
	Let $z \geq 2$ which is not a multiple of $3$, and let $d(z) = z^2 + 3$. We have 
\begin{eqnarray}\label{cor-yam-for}
	\log \eps_{d(z)} \gg \big(\log \sqf(d(z))\big)^2.
\end{eqnarray}
\end{corollary}
\begin{proof}
We choose $\alpha = 2z$ and $p = 3$ in Theorem \ref{thm-Yam}. When $3 \nmid z$, we know that $p$ splits in the real quadratic field $K_z = \QQ(\sqrt{4z^2 + 12}) = \QQ(\sqrt{d(z)})$. Therefore, by (\ref{tool-1}) and Theorem \ref{thm-Yam}, we get
	$$\log \eps_{d(z)} \geq \log \eps_{K_z} \gg (\log \Delta_{K_z})^2 \gg \big(\log \sqf(d(z))\big)^2.$$
\end{proof}

Note that by using Corollary \ref{cor-yam} and the class number formula (\ref{cnf}), we may obtain the upper bound (\ref{Yam-tri}) for the average of class numbers.

\begin{remark}
	The lower bound (\ref{cor-yam-for}) is indeed sharp. For example, Golubeva \cite{Gol1} has shown that
	\begin{eqnarray*}
		\eps_{d(z)} \leq 2 \, \bigg( \Big(\frac{z+\sqrt{d(z)}}{3} \Big)^n \Big(\frac{2+\sqrt{d(z)}}{z+1} \Big) \Big(\frac{z-1+\sqrt{d(z)}}{2} \Big) \bigg)^2,
	\end{eqnarray*}
	for $z = 3^n + 1$.
\end{remark}

We also need the following lemma on the square-free part of polynomial values.
\begin{lemma}\label{lemma-LS}
	Let $f \in \ZZ[x]$ be a separable polynomial with degree $\geq 2$. Then for almost all $z \in \ZZ_{>0}$, we have
	\begin{equation*}
		\sqf(f(z)) \geq z^{\frac{2}{3} - \epsilon},
	\end{equation*}
	for any $\epsilon > 0$.
\end{lemma}
\begin{proof}
    We consider the quantity 
    \begin{eqnarray*}
    	Q_f(S,Z) = \#\{ (z,r,s) \in \ZZ^3: 1 \leq z \leq Z, 1 \leq s \leq S, f(z) = sr^2 \}.
    \end{eqnarray*}
	Let $f(x) \in \ZZ[x]$ be a polynomial satisfying our assumptions. By a result of Luca--Shparlinski \cite[Theorem 1.3]{LS}, we have
	\begin{eqnarray*}
		Q_f(S,Z) \ll_f Z^{\frac{1}{2} + \epsilon}S^{\frac{3}{4}},
	\end{eqnarray*}
	for any $\epsilon > 0$.
	We may choose $S = Z^{\frac{2}{3} - 2\epsilon}$, so that
	$Q_f(S, Z) = o(Z)$. This implies that for almost all $z \leq Z$, we have $\sqf(f(z)) \geq Z^{\frac{2}{3} - \epsilon} \geq z^{\frac{2}{3}-\epsilon}$, for any $\epsilon > 0$.
\end{proof}

\begin{lemma}\label{lemma-sf}
	Let $f \in \ZZ[x]$ be a non-constant polynomial, and let $g(x) = f(x)^2 + 3$. Then for almost all $z \in \ZZ_{>0}$, and so we have
	$$\log \sqf(g(z)) \gg \log z,$$
	where the implied constant depends on the choice of $f$.
\end{lemma}
\begin{proof}
	We may decompose 
	$$g(x) = c \cdot h_1(x) \cdot h_2(x)^2,$$
	where $c \in \ZZ$, $h_1, h_2 \in \ZZ[x]$, and $h_1$ separable.
	
	We shall firstly show that $\deg(h_1) > 0$.
   If not, then we may consider the formula
   $$f(z)^2 + 3 = g(z) = c \cdot h_2(z)^2$$
   for integer variable $z$.
   By comparing the leading coefficients of both sides, we know $c$ is a perfect square, hence without loss of generalization, we may assume that $c = 1$. However, the only integer solution to the equation $z_1^2+3 = z_2^2$ is that $z_1 = \pm 1$ and $z_2 = \pm 2$, contradicting the assumption that $f$ is a non-constant polynomial.
   
   Let $g_1 = c \cdot h_1(z)$. Since $2 \, | \, \deg(g)$ by its definition, we know that  $\deg h_1 \geq 2$. 
   By Lemma \ref{lemma-LS}, we know $\sqf(g_1(z)) \gg z^{\frac{2}{3} - \epsilon}$ for almost all $z \in \ZZ_{> 0}$. Note that $\sqf(g(z)) = \sqf(g_1(z))$, we have
   \begin{equation*}
   	\log \sqf (g(z)) = \log \sqf (g_1(z)) \gg \log z,
   \end{equation*}  
   for almost all $z \in \ZZ_{>0}$.
\end{proof}

Now we are ready to give the proof of Proposition \ref{A1-sol}.

\begin{proof}[Proof of Proposition \ref{A1-sol}]
	If the integer solution $(t, u, z_0)$ lie on an $\AAA^1$-curve defined over $\ZZ$, then there exists non-constant polynomials $t, u, d, f \in \ZZ[x]$, such that 
	$$f(0) = z_0, \quad d(x) = f(x)^2 + 3 \quad \text{ and } \quad t(x)^2 - d(x)u(x)^2 = 1.$$ 
Hence we have
 $$\lim_{x \rightarrow \infty} \frac{\log t(x)}{\log d(x)} = \lim_{x \rightarrow \infty}\frac{\deg(t) \log x}{\deg(d) \log x} = \frac{\deg(t)}{\deg(d)} < \infty.$$ 
 
If $u \neq 0$, then by the structure of the integer solutions to the Pell equation, we must have $2|t(z)| \geq 2 t_1(z) = \eps_{d(z)} + \eps_{d(z)}^{-1} > \eps_{d(z)}$ for $z \in \ZZ_{> 0}$. Hence by Corollary \ref{cor-yam-for}, Lemma \ref{lemma-LS} and Lemma \ref{lemma-sf}, we have
\begin{align*}
	\limsup_{\substack{z \rightarrow \infty \\ 3 \, \nmid \, f(z)}} \Big| \frac{\log t(z) }{\log d(z)} \Big|  \gg &  \ \limsup_{\substack{z \rightarrow \infty \\ 3 \, \nmid \, f(z)}} \frac{\log  \eps_{d(z)}}{\log d(z)}  \\
\gg & \ \limsup_{\substack{z \rightarrow \infty \\ 3 \, \nmid \, f(z)}} \frac{\big(\log \sqf(d(z))\big)^{2}}{\deg(d) \log z} \\
\gg & \ \limsup_{\substack{z \rightarrow \infty \\ 3 \, \nmid \, f(z)}} \frac{(\log z)^2}{\log z} \, \rightarrow \, \infty.
\end{align*}

Therefore, we require that
	$3 | f(z) \text{ for all sufficiently large } z,$ in particular $3 | f(3z)$ for all sufficiently large $z$, which implies $3 | f(0)$. This contradicts the assumption $3 \nmid z_0$.
\end{proof}

\subsection{Computation of the exponent}\label{PicU}
Recall the definitions of the quantities $A$, $\rho_U$ and $b$ in the Conjecture \ref{BW-h}.
Let $X = V(2uyz - y^2v+3u^2v+v^3) \subset \PP_\QQ^3$ be the completion of $U$, which is a singular cubic surface.
Let $L_1 = V(y, v), L_2 = V(u, v), L_3 = V(z, v)$. The divisor at infinity is $D = X \setminus U = V(2uyz)$, a union of these three lines. It follows that
\begin{equation}\label{com-b}
	b = 2.
\end{equation}

There is a unique singularity $[u:y:z:v] = [0:0:1:0]$ on $X$. It is of the type $\mathbf{A}_2$ according to the classification by Bruce and Wall \cite[Lemma 3]{BW}.
Let $\tilde{X} \subset \PP_{}^3 \times \PP_{[u_1:y_1:v_1]}^2$ be the blow up of $X$ at this singularity.
There are two exceptional curves
$$E_1 = ([0:0:1:0],[u_1:0:v_1]), \text{ and } E_2 = ([0:0:1:0],[0:y_1:v_1]).$$
Let $f: \tilde{X} \rightarrow X$ be the minimal desingularization map. Since the singularity of type $\mathbf{A}_2$ is a rational double point, $f$ is crepant. Therefore $K_{\tilde{X}} = f^* K_X$ and $K_{\tilde{X}}^2 = K_X^2 = 3$. In particular, $\tilde{X}$ is a weak del Pezzo surface of degree $3$. This gives
$\mathrm{rank} (\Pic_{\overline{\QQ}} (\tilde{X}) ) = 10 - 3 = 7.$

Let $L_4 = V(v+y, 3u-2z)$ and $L_5 = V(v+y, u)$ be two other lines in $X$, and denote by the same names the strict transforms of $L_i, 1 \leq i \leq 5$, in $\tilde{X}$. One may check that the intersection pairing of $L_1, L_2, L_3, L_4, L_5, E_1, E_2$ in $\tilde{X}$ takes the form
$$
\begin{array}{c|ccccccc}
 & L_1 & L_2 & L_3 & L_4 & L_5 & E_1 & E_2 \\
\hline
L_1 & -1 & 0 & 1 & 1 & 0 & 1 & 0 \\
L_2 & 0 & -1 & 1 & 0 & 0 & 0 & 1 \\
L_3 & 1 & 1 & -1 & 0 & 0 & 0 & 0 \\
L_4 & 1 & 0 & 0 & -1 & 1 & 0 & 0 \\
L_5 & 0 & 0 & 0 & 1 & -1 & 0 & 1 \\
E_1 & 1 & 0 & 0 & 0 & 0 & -2 & 1 \\
E_2 & 0 & 1 & 0 & 0 & 1 & 1 & -2 \\
\end{array}
$$
This is matrix has rank 7, hence it generates the geometric Picard group of $\tilde{X}$. Since all these lines are rational, it follow that $\mathrm{rank} (\Pic_{\QQ} (\tilde{X}) ) = 7.$ We also know from the intersection pairing that the five lines $L_1,L_2,L_3,E_1,E_2$ at infinity are linearly independent, hence
\begin{equation}\label{com-rho}
	\rho_U = 7 - 5 = 2.
\end{equation}
Therefore, by (\ref{com-b}) and (\ref{com-rho}), we conclude that 
\begin{equation}\label{com-A}
	A = \rho_U + b = 4.
\end{equation}
in the Conjecture \ref{BW-h} for the surface $U$ defined by (\ref{logK3-2}).

Moreover, the boundary divisor $E_1+E_2+L_1+L_2+L_3$ has anticanonical class in the Picard group, so $U$ is log K3.




\subsection{Conclusion of the proofs}

Recall that the changing of variable $y = t + u z$ to the equation (\ref{Pell-2}) gives us the surface
$$U = \{(y,u,z) \in \AAA^3: 2uyz = y^2 - 3u^2 - 1\}.$$
We may establish a lower bound for $N_U^\circ(B)$, as follows.

\begin{lemma}\label{height-func}
	We have
	\begin{equation*}
	N_U^\circ(B) \geq \# \{ z \leq B: 3 \nmid z, \mu^2(d(z)) = 1, \eps_{d(z)} \leq B \}.
\end{equation*}
\end{lemma}
\begin{proof}

Recall that all the solutions of (\ref{Pell-2}) are of the form
$$\{(t_n, u_n) \in \ZZ^2: t_n + \sqrt{d(z)}u_n = \pm \eps_{d(z)}^n, n \in \ZZ\}.$$
When $z \leq B, 3  \nmid  z, \mu^2(d(z)) = 1$, and $  \eps_{d(z)} \leq B$, we consider the integer point  $$P = (t_1 - u_1z, u_1, z) \in U.$$ We will show that the height of $P$ is less than $B$. 

We have 
$$u_1 = \frac{\eps_{d(z)} - \eps^{-1}_{d(z)}}{2 \sqrt{d(z)}} \gg z^{-1} \eps_{d(z)}.$$
By Corollary \ref{thm-Yam}, we know that $|u_1| > |z|$, for $3 \nmid z, \mu^2(d(z)) = 1$ when $z$ is sufficiently large. Note that
$$|t_1 - u_1 z| = \bigg| \, \frac{\eps_{d(z)} + \eps^{-1}_{d(z)}}{2} - \frac{\eps_{d(z)} - \eps^{-1}_{d(z)}}{2 \sqrt{d(z)}} \cdot z \, \bigg| \ll \Big(1 - \frac{z}{\sqrt{d(z)}} \Big) \eps_{d(z)} \ll z^{-2} \eps_{d(z)}.$$
This implies that $|u_1| > |t_1 - u_1 z|$ when $z$ is sufficiently large.
Therefore, if $\eps_{d(z)} \leq B$, then the height of $P$ is $u_1$, and we have 
$$u_1 = \frac{\eps_{d(z)} - \eps^{-1}_{d(z)}}{2 \sqrt{d(z)}} < \eps_{d(z)} \leq B.\vspace{-0.35em}$$ 
\end{proof}

We are now ready to complete the proofs of Theorem \ref{thm-lb}, Corollary \ref{rqf} and Theorem \ref{thm-cn}. From (\ref{com-A}), we know that the Browning--Wilsch Conjecture \ref{BW-h} predicts that 
\begin{eqnarray}\label{BW-h-U}
	N_U^\circ(B) \ll (\log B)^4.
\end{eqnarray}

\begin{proof}[Proof of Theorem \ref{thm-lb}]
	We consider the quantity
	\begin{eqnarray*}
		S(B) = \# \{ z \leq (\log B)^{4 + \epsilon}: 3  \nmid  z, \mu^2(d(z)) = 1, \log \eps_{d(z)} \leq \log B\}.
	\end{eqnarray*}
Then Lemma \ref{height-func} implies that 
\begin{align}\label{sb-lemma}
\begin{split}
		S(B) \leq & \ \# \{ z \leq B: 3  \nmid  z, \mu^2(d(z)) = 1, \log \eps_{d(z)} 
	 \leq \log B\}
	 \\  \leq & \ N_U^{\circ}(B) \ll (\log B)^4,
\end{split}
\end{align}
under the assumption of (\ref{BW-h-U}).
Therefore, $S(B)$ is a subset of $[1, (\log B)^{4 + \epsilon}] \cap \ZZ$ with zero density. As $B \rightarrow \infty$, then for all $z \leq (\log B)^{4 + \epsilon}$ outside a subset of size $o((\log B)^{4 + \epsilon})$ with $3 \nmid z$ and $z^2 + 3$ square-free, we have 
$$\log \eps_{d(z)} > \log B \geq  z^{\frac{1}{4 + \epsilon}} \gg d(z)^{\frac{1}{8} - \epsilon}. \vspace{-0.35em}$$
\end{proof}

\begin{proof}[Proof of Corollary \ref{rqf}]
     For the same reason as in Remark \ref{rmk-Ricci}, we know that there is a positive proportion of $z \in \ZZ_{> 0}$ with $z \equiv 26 \, (\mod 42)$, such that $d(z) = z^2 + 3$ is square-free. We will show that when $z$ satisfies the above assumption, the sequence of real quadratic fields $K_z = \QQ(\sqrt{d(z)})$ satisfies (\ref{rqf}), when $z \rightarrow \infty$. Note that $d(z) \equiv 3 \, \pmod 4$, and so we have $\Delta_{K_z} = 4 d(z)$.

    By assumption, we know $d(z)$ is a multiple of $7$, so the negative Pell equation $$t^2 - 4d(z) \cdot u^2 = -4$$ is not solvable for $\mathrm{mod} \ 7 $ reasons. Let $(a,b)$ be the pair of smallest positive integers satisfying
    $a^2 - 4d(z) \cdot b^2 = 4.$ Note that $a$ must be even, so we have
    $(\frac{a}{2})^2 - d(z) \cdot b^2 = 1.$
    Hence we know that
    $$\eps_{K_z} = \frac{a}{2} + b \sqrt{d(z)} \geq \eps_{d(z)}.$$
    Then we may apply Theorem \ref{thm-lb} to conclude the proof by using the assumption $3 \nmid z$ and $d(z)$ square-free.
\end{proof}

\begin{proof}[Proof of Theorem \ref{thm-cn}]
Let $Z \leq B$ be a parameter, it follows from (\ref{sb-lemma}) that
	\begin{align*}
		& \ \# \{ z \leq Z: 3  \nmid  z, \mu^2(d(z)) = 1, \log \eps_{d(z)} \leq \log B \} \\
		 \leq & \ \# \{ z \leq B: 3  \nmid  z, \mu^2(d(z)) = 1, \log \eps_{d(z)} \leq \log B \}  \ll (\log B)^4,
	\end{align*}
under the assumption of (\ref{BW-h-U}).
By the class number formula (\ref{cnf}), Theorem \ref{thm-Yam} and the bound $L_d(1) \ll \log d$, we have
	\begin{align*}
		\sum_{\substack{z \leq Z, 3 \, \nmid \, z \\ \mu^2(d(z)) = 1 \\ \log \eps_{d(z)} \leq \log B}} h(z^2 + 3)  = & \sum_{\substack{z \leq Z, 3 \, \nmid \, z \\ \mu^2(d(z)) = 1  \\ \log \eps_{d(z)} \leq \log B}} \frac{\sqrt{d(z)} L_d(1)}{\log \eps_{d(z)}} \\
		 \ll & \  (\log B)^4 \cdot \frac{Z \log Z}{(\log Z)^2} \\
		  = & \ (\log B)^4 \cdot Z \log^{-1} Z.
	\end{align*}
On the other hand, we have
	\begin{align*}
		\sum_{\substack{z \leq Z, 3 \, \nmid \, z \\ \mu^2(d(z)) = 1  \\ \log \eps_{d(z)} > \log B}} h(z^2 + 3)  = & \sum_{\substack{z \leq Z, 3 \, \nmid \, z \\ \mu^2(d(z)) = 1  \\ \log \eps_{d(z)} > \log B}} \frac{\sqrt{d(z)} L_d(1)}{\log \eps_{d(z)}} \\
		 \ll & \  \sum_{z \leq Z} \frac{z \log z}{\log B} \\
		 \ll &  \  (\log B)^{-1} \cdot Z^2 \log Z.
	\end{align*}
Finally, we choose $\log B = (Z \log^2 Z)^{\frac{1}{5}}$ to conclude the proof of Theorem \ref{thm-cn}.
\end{proof}

\end{document}